\documentclass{article}
\usepackage{amsmath}
\usepackage{amsthm}
\usepackage{graphicx}
\usepackage{geometry}
\usepackage{hyperref}
\usepackage{booktabs}
\usepackage{amsfonts} 
\theoremstyle{plain}
\newtheorem{proposition}{Proposition}

\usepackage{pdfpages}

\title{A golden–ratio partition of information and the balance between prediction and surprise: a neuro–cognitive route to antifragility}
\author{Pablo Padilla$^{1}$ \and Oliver López-Corona$^{2,*}$ \and Elvia Ram\'irez-Carrillo$^{3}$ \and Ariadne Hern\'andez S\'anchez$^{4}$}
\date{\today}

\begin{document}
\maketitle
\vspace{-0.8cm}
\begin{center}
\small $^{1}$Instituto de Investigaciones en Matemáticas Aplicadas y Sistemas (IIMAS), Universidad Nacional Autónoma de México (UNAM), Mexico City, Mexico
\small $^{2}$Investigadores por M\'exico (IxM-Secihti), Instituto de Investigaciones en Matemáticas Aplicadas y Sistemas (IIMAS), Universidad Nacional Autónoma de México (UNAM), Mexico City, Mexico
\small $^{3}$Facultad de Ciencias, Universidad Nacional Autónoma de México (UNAM), Mexico City, Mexico
\small $^{4}$ Facultad de Estudios Superiores Iztacala, Universidad Nacional Autónoma de México (UNAM), Mexico City, Mexico
\small $^{*}$Corresponding author: oliver.lopez@iimas.unam.mx
\end{center}

\begin{abstract}

Adaptive systems must strike a balance between prediction and surprise to thrive in uncertain environments. We propose an information-theoretic balance function, $ f(p) = -(1 - p)\ln(1 - p) + \ln p $, which quantifies the net informational gain from contrasting explained variance $p$ with unexplained novelty $(1 - p$. This function is strictly concave on $(0,1)$ and reaches its unique maximum at $ p^* \approx 0.882$, revealing a regime where confidence is high but the residual uncertainty carries a disproportionate potential for surprise.

Independently of this maximum, imposing a self-similarity condition between known, unknown and total information, $p : (1-p) = 1 : p$, leads to the golden-ratio reciprocal $p = 1/\varphi \approx 0.618$, where $ \varphi$ is the golden ratio. We interpret this value not as the maximizer of  $f$, but as a structurally privileged \emph{partition} in which known and unknown are proportionally nested across scales.

Embedding this dual structure into a Compute-Inference-Model-Action (CIMA) loop yields a dynamic process that maintains the system near a critical regime where prediction and surprise coexist. At this edge, neuronal dynamics exhibit power-law structure and maximal dynamic range, while the system’s response to perturbations becomes convex at the level of its payoff function-fulfilling the formal definition of antifragility. We suggest that the golden-ratio partition is not merely a mathematical artifact, but a candidate design principle linking prediction, surprise, criticality, and antifragile adaptation across scales and domains, while the maximum of $f$ identifies the point of greatest informational vulnerability to being wrong.
\end{abstract}

\section*{An Informational Balance Equation}

This manuscript presents a theoretical contribution aimed at clarifying general principles by which biological systems regulate the balance between prediction and surprise under uncertainty. By introducing a formally defined information-theoretic balance function and identifying structurally privileged partitions of explained and unexplained variance, the work addresses a central question in theoretical biology and neuroscience: how adaptive systems maintain sensitivity to novelty while preserving internal coherence. The framework is intentionally abstract but grounded in biologically interpretable quantities (see supplementary materials) such as prediction error, explained variance, precision weighting, and learning progress, allowing the proposed constructs to be mapped onto neural and cognitive processes without reliance on detailed mechanistic assumptions \cite{friston2010}.

Let us consider the following informational balance equation (with ln the natural logarithm)

\begin{equation}\label{eq:balance}
f:(0,1)\to\mathbb{R}, \qquad
f(p) = -(1 - p)\ln(1 - p) + \ln p .
\end{equation}

\begin{proposition}[Strict concavity]
The function $f$ is strictly concave on $(0,1)$.
\end{proposition}
\begin{proof}
A direct computation gives
\[
f''(p) = -\frac{1}{1-p} - \frac{1}{p^2},
\]
which is strictly negative for all $p\in(0,1)$. Hence $f$ is strictly concave on $(0,1)$.
\end{proof}

We may interpret $f$ using cognitive processes in a three-layer analogy:

\begin{table}[h!]
\centering
\small
\caption{Three-layer analogy connecting probability, information, and cognitive meaning.}
\begin{tabular}{@{}p{3cm}p{3cm}p{5cm}p{5.5cm}@{}}
\toprule
\textbf{Layer} & \textbf{Mathematical object} & \textbf{Standard meaning} & \textbf{Metaphorical / cognitive reading} \\
\midrule
\textit{1. A priori beliefs} & $p \in (0,1)$ & Probability assigned to a specific event $A$ & “What we believe we know” (our confidence in $A$) \\
 & $1 - p$ & Probability of the complement $A^c$ & “What we acknowledge we do not know” (confidence that something different from $A$ may occur) \\
\addlinespace
\textit{2. Post-observation surprise} & $S(A) = -\ln p$ & Self-information when $A$ occurs & How surprising it is to confirm what we expected \\
 & $S(A^c) = -\ln(1 - p)$ & Self-information when $A^c$ occurs & How surprising it is to witness what we had ruled out \\
\addlinespace
\textit{3. Average uncertainty} & $H = -[p\ln p + (1 - p)\ln(1 - p)]$ & Shannon entropy for a Bernoulli variable & “How much surprise we still expect on average”: a blend of surprises from both what we think we know and what we admit we don’t \\
\bottomrule
\end{tabular}
\end{table}

Then the function from Eq.~\eqref{eq:balance} can be decomposed into two terms, each carrying a specific probabilistic meaning and a cognitive/metaphorical interpretation:

\begin{table}[h!]
\centering
\small
\caption{Components of the balance function and their interpretations.}
\begin{tabular}{@{}p{3.2cm}p{3.5cm}p{4.2cm}p{5.5cm}@{}}
\toprule
\textbf{Term} & \textbf{Formula} & \textbf{Probabilistic meaning} & \textbf{Metaphorical / cognitive reading} \\
\midrule
\textbf{A. Expected surprise from the complement} & $-(1 - p)\ln(1 - p)$ & Probability of the complement $A^c$ weighted by its surprise: $-\ln(1 - p)$. This is the average information we would gain if the unexpected occurs. & “How much surprise we can expect from the part of reality we admit not to control.” \\
\addlinespace
\textbf{B. (Negative) surprise from the expected event} & $+\ln p$ & The surprise of observing the expected event $A$ is $-\ln p$. This term appears with a positive sign, effectively subtracting the surprise of confirming our expectations. & “We discount the surprise of the known; the greater our confidence in $A$, the smaller the penalty for its confirmation.” \\
\bottomrule
\end{tabular}
\end{table}

or in other words,

\[
f(p) = 
\underbrace{\text{Expected surprise from the unknown}}_{\text{(1)}} 
\;-\; 
\underbrace{\text{Pointwise surprise from the known}}_{\text{(2)}}
\]

This function measures the \textbf{net informational payoff}---how much more (or less) surprise is expected from what we do not know compared to what we believe we know.

\begin{itemize}
    \item \textbf{If $f(p) > 0$}: The \emph{surprise potential} of the unknown ($A^c$) exceeds the cost of confirming the expected ($A$). Our confidence could become a liability.
    \item \textbf{If $f(p) < 0$}: The \emph{self-inflicted surprise} of confirming our own expectations dominates; the unknown adds little informational gain.
\end{itemize}

\subsection*{Behavior at the extremes}

\begin{table}[h!]
\centering
\small
\caption{Extreme behavior of $f(p)$ and its interpretation.}
\begin{tabular}{@{}p{3cm}p{3cm}p{8cm}@{}}
\toprule
\textbf{Limit} & \textbf{Value of $f(p)$} & \textbf{Interpretation} \\
\midrule
$p \to 0$ & $f(p) \to -\infty$ & We are almost certain that $A$ will \emph{not} occur. If it does, $\ln p \to -\infty$, producing an enormous shock that dominates the balance. \\
\addlinespace
$p \to 1$ & $f(p) \to 0^+$ & The complement becomes almost impossible; its expected surprise vanishes. Confirming $A$ also adds nearly no new information. \\
\addlinespace
$p \approx 0.7$ & $f(p) \approx 0$ & The expected surprise from the unknown equals the cost of confirming the known. This is the equilibrium point where neither dominates. \\
\bottomrule
\end{tabular}
\end{table}

In this way, term A quantifies the degree of potential surprise that may still arise from the domain we explicitly recognize as uncertain. It reflects the information content of those outcomes we assign low probability to—the events we believe are unlikely but nonetheless possible. This term represents the informational cost of being proven wrong about what we assumed was improbable, and as such, it captures the vulnerability embedded in our admitted ignorance.

In contrast, Term B represents a discount applied to the surprise that would arise from confirming what we already expect. It may seem paradoxical, but even when an event we are confident about does occur, it still carries an informational signature: the confirmation is not completely devoid of surprise, though its magnitude diminishes as our confidence grows. This term subtracts that residual, expected surprise, reminding us that confidence is never free of informational cost.

Together, these two terms combine to define the function $f(p)$, which expresses the net informational balance between unknowns and knowns. The sign and magnitude of $f(p)$ indicate which side of our epistemic model contributes more to our potential shock. A positive value signals that the neglected or underestimated portion of the probability space could surprise us more than the confirmation of our beliefs, while a negative value suggests that our own overconfidence may be our primary source of informational risk. This balance is central to understanding how systems allocate attention, manage expectations, and navigate environments filled with both structure and uncertainty.

In Fig. 1 we present a conceptual summary figure to guide the reading of the rest of the manuscript. 

\begin{figure}
    \centering
    \includegraphics[width=0.9\linewidth]{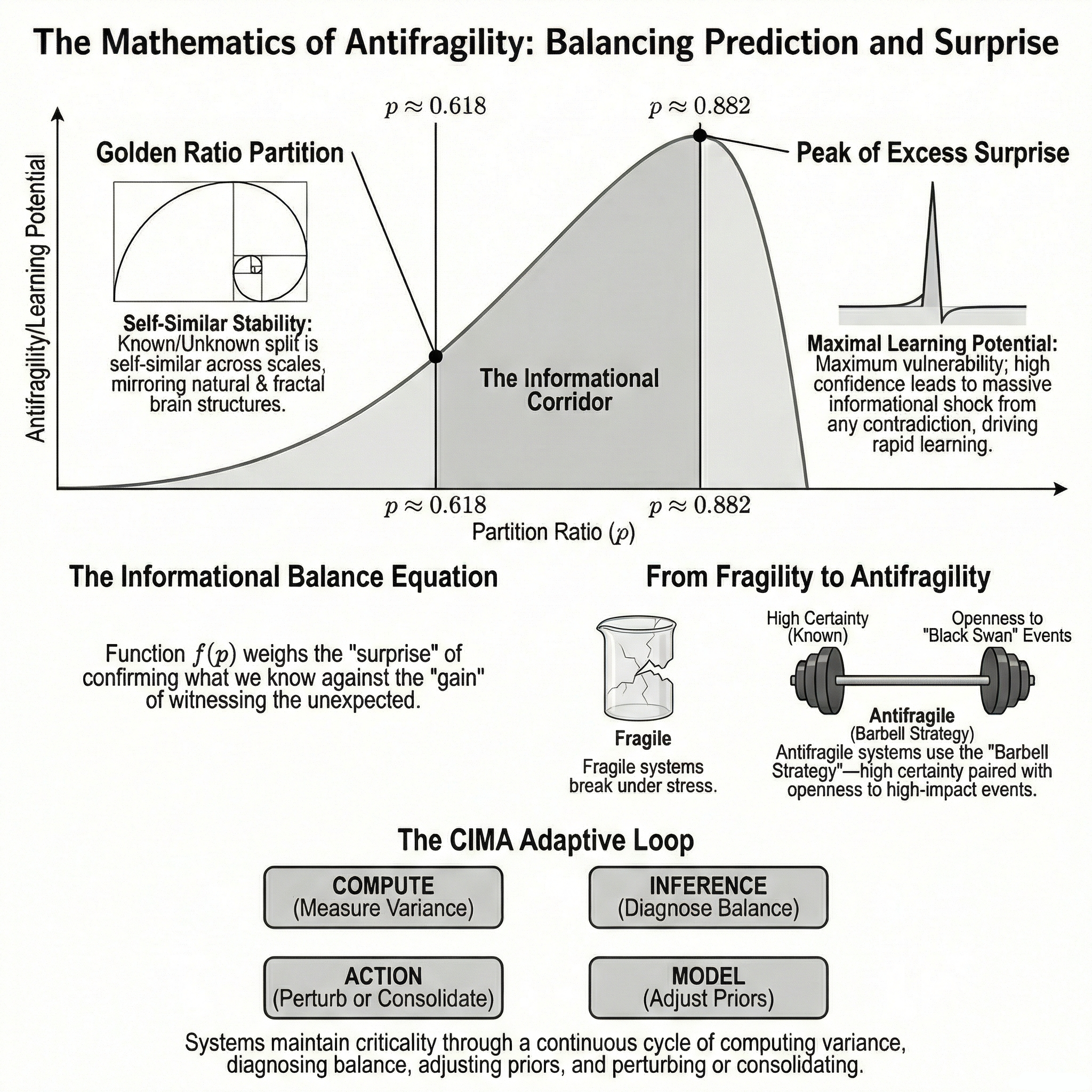}
    \caption{The curve represents the informational balance function 
$ f(p) = -(1-p)\ln(1-p) + \ln p $, where $ p \in (0,1) $ denotes the proportion of variance explained by an internal model (prediction) and $1-p$ the residual unexplained component (novelty). 
The function quantifies the net informational imbalance between the expected surprise arising from the unknown component and the verification cost associated with confirming what is already predicted. 
The unique maximum at $ p^* \approx 0.882 $ identifies the regime in which residual uncertainty exerts the greatest asymmetric informational influence, corresponding to a state of high confidence coupled with maximal vulnerability to being wrong. 
Independently, the self-similar partition $ p_\phi = 1/\varphi \approx 0.618 $, derived from the relation $ p:(1-p)=1:p $, marks a structurally distinguished informational split in which the ratio of known to unknown equals the ratio of known to the whole. 
The interval between $ p_\phi $ and $ p^* $ defines an informational corridor in which predictive structure and novelty coexist, providing the theoretical basis for adaptive operation near criticality, where systems remain coherent yet sensitive to perturbations.
}
\end{figure}

\section{Maximisation and golden self-similarity}\label{sec:max}

Consider the informational balance function. To locate its internal maximum, we compute its first derivative:
\[
f'(p) = 1 + \ln(1 - p) + \frac{1}{p}.
\]

Setting the derivative equal to zero yields the critical point:
\[
f'(p) = 0 \quad \Longleftrightarrow \quad 1 + \ln(1 - p) + \frac{1}{p} = 0.
\]

\begin{proposition}[Existence and uniqueness of the maximizer]
There exists a unique $p^{\star}\in(0,1)$ such that $f'(p^{\star})=0$; moreover, this $p^{\star}$ is the unique global maximizer of $f$ on $(0,1)$.
\end{proposition}
\begin{proof}
Since
\[
f''(p) = -\frac{1}{1-p} - \frac{1}{p^2} < 0 \quad \text{for } 0<p<1,
\]
the derivative $f'$ is strictly decreasing on $(0,1)$. In addition,
\[
\lim_{p\to 0^+} f'(p) = +\infty,\qquad \lim_{p\to 1^-} f'(p) = -\infty.
\]
By the intermediate value theorem there is at least one $p^{\star}\in(0,1)$ with $f'(p^{\star})=0$, and strict monotonicity of $f'$ implies it is unique. Strict concavity then implies this critical point is the unique global maximizer.
\end{proof}

Solving the transcendental equation $f'(p)=0$ numerically gives a unique root in the interval $0<p<1$:
\[
p^{\star} \approx 0.882 \quad \text{(to three decimals)}.
\]

Evaluating the function at this point yields:
\[
f(p^{\star}) \approx 0.127.
\]

In the following table we present our interpretation of the maximum: 

\begin{table}[h!]
\centering
\small
\caption{Interpretation of the critical point of $f(p)$.}
\begin{tabular}{@{}p{3.5cm}p{4.5cm}p{6.5cm}@{}}
\toprule
\textbf{Feature} & \textbf{Formal reading} & \textbf{Cognitive/metaphorical reading} \\
\midrule
$p^{\star} \approx 0.882$ & High (but not total) confidence in event $A$ & We believe we know almost everything, leaving just 12\% of uncertainty—only a small “grey zone.” \\
\addlinespace
$f(p^{\star}) > 0$ & Average surprise from the residual 12\% outweighs the surprise of confirming what we expected & We are maximally vulnerable to being blindsided: confidence is so high that contradiction from reality would cause a disproportionately large shock. \\
\addlinespace
$f'(p^{\star}) = 0$ & Function is flat: any small change in $p$ reduces net surprise & This is a “peak imbalance” point between known and unknown. Any move toward more or less certainty diminishes the tension. \\
\bottomrule
\end{tabular}
\end{table}

Imagine a risk engineer who estimates a 88.2\% chance that a system will function correctly, and a 11.8\% chance it will fail.

If the system works, which is likely, the informational gain is modest:
\[
-\ln(0.882) \approx 0.126.
\]

If it fails, which is unlikely, the informational cost is large:
\[
-\ln(0.118) \approx 2.14,
\]
and its contribution, weighted by probability, is:
\[
0.118 \cdot 2.14 \approx 0.25.
\]

The net difference, given by the balance function, is:
\[
f(p^{\star}) \approx 0.25 - 0.126 = 0.127.
\]

This represents the point of maximum excess surprise: a situation where confidence borders on complacency, and the impact of being wrong peaks despite the high probability of being right.

This scenario can be directly connected to Taleb's barbell strategy. In \emph{Antifragile}, Taleb proposes that the best way to navigate uncertainty is not by relying on moderate, average-risk options, but by combining two extremes: placing most of one’s resources in extremely safe, low-risk domains, while simultaneously allocating a small portion to highly volatile, high-uncertainty exposures with large potential upside. The informational model presented here reflects the same logic. Consider a risk engineer who assigns a probability of $p = 0.882$ to the system functioning correctly, and $1 - p = 0.118$ to failure. If the system performs as expected, the information gained is modest, since $-\ln(0.882) \approx 0.126$. However, if the unlikely event occurs and the system fails, the informational cost is substantial: $-\ln(0.118) \approx 2.14$. Even after being weighted by its small probability, the contribution of this event is $0.118 \cdot 2.14 \approx 0.25$, which dominates the expected gain from the correct prediction. The net balance, given by the informational function $f(p)$, is therefore $f(p^{\star}) \approx 0.127$, marking the point of maximum excess surprise. This reflects a cognitive structure identical to the barbell principle: high certainty about the main outcome (safe side), paired with an open window to tail events (risky side) whose informational payoff is convex. In this regime, confidence is high but not absolute, and any contradiction from reality results in an informational shock disproportionate to the gain of being correct. This is not fragility, but antifragility—where the cost of being right is low, and the gain from being wrong is high. Just as the barbell strategy converts volatility into advantage, this probabilistic configuration converts epistemic surprise into learning.

\subsection{Golden self-similarity as an informational partition}

Beyond the location of the maximum of \( f \), we may ask whether there exists a structurally special value of \( p \in (0,1) \) at which the relation between what is known, what is unknown, and the whole exhibits a self-similar proportion. A natural requirement is that
\[
\frac{p}{1-p} = \frac{1}{p},
\]
that is:
\[
p : (1-p) = 1 : p.
\]
This expresses the idea that the ratio of known to unknown equals the ratio of known to total: the structure of the epistemic split repeats itself when zooming between scales.

Solving
\[
\frac{p}{1-p} = \frac{1}{p}
\quad \Longleftrightarrow \quad
p^2 + p - 1 = 0
\]
yields
\[
p_\varphi = \frac{\sqrt{5} - 1}{2} = \frac{1}{\varphi} \approx 0.618,
\]
where
\[
\varphi = \frac{1 + \sqrt{5}}{2} \approx 1.618,
\qquad
\varphi^2 = \varphi + 1.
\]

We can now interpret each step of this derivation through the lens of our “known vs. unknown” epistemic framework:

\begin{table}[h!]
\centering
\small
\caption{Narrative interpretation of the golden-ratio self-similar partition.}
\begin{tabular}{@{}p{1.5cm}p{5cm}p{3.2cm}p{6cm}@{}}
\toprule
\textbf{Step} & \textbf{Operation} & \textbf{Result} & \textbf{Narrative interpretation} \\
\midrule
1 & Impose self-similar partition & $p : (1-p) = 1 : p$ & The proportion between known and unknown mirrors the proportion between known and total information. \\
\addlinespace
2 & Rewrite as an equation & $\dfrac{p}{1-p} = \dfrac{1}{p}$ & We require that the epistemic split remains invariant under “zooming out” from the unknown domain to the total. \\
\addlinespace
3 & Solve algebraically & $p^2 + p - 1 = 0$, $p = 1/\varphi \approx 0.618$ & We obtain the golden-ratio partition: a fixed point of recursive balance between knowledge and novelty. \\
\bottomrule
\end{tabular}
\end{table}

Of course, the golden ratio \( \varphi \approx 1.618 \) itself cannot be a probability (as it exceeds 1). However, its reciprocal lies within the admissible domain \( (0,1) \) and inherits the same self-similarity:
\[
\frac{1}{\varphi} = \varphi - 1 \approx 0.618.
\]

We therefore interpret
\[
p_\varphi \approx 0.618
\]
not as the maximizer of the balance function \( f \), but as a structurally distinguished partition: a recursively self-similar split between known and unknown information. The maximum \( p^{*} \approx 0.882 \) and the golden partition \( p_\varphi \approx 0.618 \) thus play complementary roles: the former identifies the point of maximum excess surprise, while the latter defines a canonical self-similar balance structure.

This auto-similar proportion is the same that underlies the Fibonacci sequence and many natural patterns of growth, structure, and harmony. Here it emerges as an informational fixed point in the space of epistemic partitions.

\section{Criticality correspondence}\label{sec:crit}

In complex systems, especially the brain, criticality manifests as a state of marginal stability between two opposing forces: self-organization and emergence. This regime is empirically characterized by scale-invariant dynamics, such as neuronal avalanches, $1/f$ electrical spectra, and fractal morpho-architecture \cite{beggs2003,bak1987,chialvo2010,beggs2008}. These phenomena obey power laws, such as $P(s) \propto s^{-\tau}$ for avalanche sizes and durations, and remain invariant under changes in temporal or spatial resolution \cite{clauset2009}. In other words, criticality expresses a deep structural \emph{self-similarity}.

This self-similarity is not merely mathematical; it echoes the same informational proportion captured by our partitions of \( p \). The maximum of \( f(p) \) at \( p^{*} \approx 0.882 \) shows where the residual uncertainty exerts maximal asymmetric influence on the informational balance. The golden-ratio partition \( p_\varphi \approx 0.618 \) selects a self-similar split where the ratio of known to unknown equals the ratio of known to the whole. In a $1/f$ spectrum, an analogous recursive balance holds: zooming in or out yields the same statistical structure. This reflects a nested tension between stability and novelty, prediction and surprise.

Empirical evidence supports the idea that the brain self-organizes around a critical zone. Fractal brain structures have been observed across species (human, mouse, and fly), suggesting a universal structural transition point. Neuronal avalanches modulate with sleep phases, and deviations from criticality, such as the spectral slope shifting away from $1/f$, have been linked to neurological disorders like Alzheimer’s disease. Moreover, experiments show that injecting endogenous-like $1/f$ noise into artificial and biological neural networks can synchronize their activity near this critical boundary.

From our narrative perspective, we can map the roles as follows: the \emph{self-organizing} side is driven by stable attractor dynamics (oscillations, synchrony) and corresponds to “the known” in our model, i.e., internal predictions and priors. It reduces dimensionality, saves energy, and promotes local coherence. The \emph{emergent} side corresponds to “the unknown”, external fluctuations, surprises, and deviations that produce avalanches and scale-free activity. This component increases representational diversity and responsiveness. The \emph{critical point}, where the system explores a corridor between \(p_\varphi\) and \(p^{*}\), represents an optimal balance: neither frozen in predictability nor lost in chaos.

Why would a brain want to operate near this boundary? First, it enables a \textbf{maximal dynamic range}, allowing neurons to respond proportionally to stimuli across multiple scales \cite{kinouchi2006,shew2009,shewplenz2013}. Second, it supports \textbf{predictive power}: the self-organized patterns implement efficient internal models, while the emergent fluctuations inject prediction errors needed for learning \cite{beggs2008,shewplenz2013}. Third, it promotes \textbf{energetic efficiency}: critical avalanches minimize redundant spikes yet ensure coverage of the state space \cite{beggs2008}. Finally, it allows \textbf{robust flexibility}: slight changes in synaptic gain can move the network between subcritical, critical, and supercritical modes, enabling adaptive shifts between exploration and exploitation \cite{beggs2008,hessegross2014}.

In short, a brain operating near criticality physically realizes the very dialectic we have formalized: known and unknown, order and surprise, recursively mixed in a scale-free proportion. This balance confers the optimal blend of stability and adaptability demanded by any organism that must both anticipate the world and remain sensitive to what it has never seen before.

Most important, criticality has been proposed as the first real universal principle for complex systems: the \textbf{Principle of Dynamic Criticality (PDC)} as the hypothesis that living systems, under evolutionary pressure, tend to self-organize near critical points of their dynamical phase space, balancing order and chaos. At these points, known as \emph{dynamically critical states}, systems exhibit scale-invariant behavior, maximal dynamic range, and enhanced computational and adaptive capabilities. This principle suggests that criticality is not merely a structural or statistical feature, but a functional attractor favored by natural selection due to its capacity to optimize sensitivity to perturbations, efficiency in information processing, and responsiveness to environmental change.

Under this view, dynamic criticality emerges as a unifying macroscopic operating principle of life, analogous in generality to the principle of evolution. It provides a theoretical anchor to interpret the recurrence of $1/f$ spectra, neuronal avalanches, and other power-law phenomena observed across biological scales. Importantly, it is also applicable beyond biology: any adaptive system—natural or artificial—may benefit from operating at the edge of criticality, where known structures and unknown fluctuations coexist in an optimally balanced, self-similar proportion.

We highlight that extensive evidence points to systems at criticality being at a maximum of complexity, computational and inferential capabilities. This will be of special importance in the next section.

\section{Embedding in the CIMA loop}\label{sec:cima}

A \textbf{CIMA process} is a recursive, four-stage adaptive cycle composed of \emph{Compute}, \emph{Inference}, \emph{Model}, and \emph{Action}, by which a system dynamically calibrates the balance between prediction and surprise. The goal of the cycle is to maintain the system near a critical regime—neither too ordered nor too chaotic—where informational payoff and adaptability are maximized \cite{gershenson2005}.

\begin{enumerate}
    \item \textbf{Compute}: The system estimates the proportion $p$ of variance explained by its internal model relative to total observed variance. This quantifies the current balance between the known and the unknown.
    
    \item \textbf{Inference}: Using the estimated $p$, the system evaluates a balance function—e.g., $f(p) = -(1 - p)\ln(1 - p) + \ln p$—alongside independent dynamical markers (e.g., $1/f$ scaling, avalanche exponents) to diagnose whether it is near a desirable operating regime.
    
    \item \textbf{Model}: Based on this diagnosis, the system reconfigures its internal priors or learning parameters to increase or decrease model confidence. This may involve adjusting precision weighting, model complexity, or gain modulation to steer the system toward an appropriate balance.
    
    \item \textbf{Action}: The system acts on the environment—or itself—either to consolidate predictions (if overexposed to surprise) or to introduce controlled perturbations (if overconfident), thereby maintaining dynamic criticality.
\end{enumerate}

\noindent
A CIMA process thereby implements a continuous loop of \textbf{adaptive inference under uncertainty}, where antifragility emerges from the system’s ability to transform perturbations into information gains. It provides a general computational scaffold for critical cognition in biological systems, and a blueprint for designing artificial systems that learn and thrive at the edge of chaos.

As we have shown above, the informational balance function
\[
f(p) = -(1 - p)\ln(1 - p) + \ln p
\]
has a unique maximum at \( p^{*} \approx 0.882 \), identifying the configuration in which residual uncertainty has the greatest asymmetric impact. Independently, the condition of self-similar partition between known, unknown, and whole,
\[
\frac{p}{1-p} = \frac{1}{p},
\]
selects the golden-ratio reciprocal point \( p_\varphi = 1/\varphi \approx 0.618 \). This value represents a self-similar partition of information: the ratio of known to unknown is equal to the ratio of known to the whole, i.e., \( p : (1 - p) = 1 : p \). This recursive structure mirrors the internal consistency of the Fibonacci sequence and $1/f$ dynamics observed in critical systems. When embedded into the CIMA loop, this partition can be used as a \emph{design rule} for self-regulation across the four stages of adaptive cognition, while \( p^{*} \) marks the point of maximal vulnerability if the system drifts too far into overconfidence.

\textbf{Compute.} The system begins by quantifying how much of reality is currently explained by its internal model. This is formalized as \( p = \mathrm{Var}_{\text{explained}}/\mathrm{Var}_{\text{total}} \), while the residual \( 1 - p \) captures what remains surprising. A system targeting \( p \approx p_\varphi \) maintains a robust mix of model certainty and openness to novelty.

\textbf{Inference.} The system evaluates the differential surprise between known and unknown using the components of the balance function: \( S_{\text{unknown}} = -(1 - p)\ln(1 - p) \) and \( S_{\text{known}} = -\ln p \). When \( f(p) \approx 0 \), the surprise expected from the unexplained balances the confirmation cost of what is known. This condition holds near a broad intermediate range (including \( p \approx 0.7 \)), and deviations from it can be used as a continuous signal to detect whether the system is drifting toward overconfidence (high $p$) or confusion (low $p$).

\textbf{Model.} Based on this inference, the model is updated to preserve an adequate proportion between known and unknown. One particularly elegant choice is the golden partition \( p_\varphi \), where about 62\% of representational capacity focuses on refining the already captured structure (exploitation), while 38\% remains dedicated to characterizing residual unpredictability (exploration). This adjustment ensures that the system “grows into itself” without losing adaptability—a hallmark of self-similarity.

\textbf{Action.} Finally, the system acts on itself or the environment to maintain dynamic balance. If \( p \ll p_\varphi \), the system collects more data or amplifies perturbations (e.g., stress-testing, exploratory moves). If \( p \gg p_\varphi \), it introduces noise, performs A/B testing, or increases uncertainty to prevent overfitting. These interventions tune the exploration–exploitation trade-off, keeping the system near the critical regime.

\textbf{Systemic reading.} The golden-ratio partition at \( p_\varphi = 1/\varphi \) provides a self-similar architecture across scales: zooming from total information → known → unknown preserves the same ratio. This echoes the scale-invariant structure of $1/f$ spectra and critical networks. Operating near such a partition grants the system maximal responsiveness to new signals without sacrificing internal coherence—key traits of critical dynamics in neural, ecological, and social systems.

\textbf{Operational antifragility.} Remaining within an “informational corridor” between too little and too much certainty prevents both rigid overconfidence (too much order) and unstable chaos (too much noise), positioning the system to improve from perturbations. In practical terms: estimate \( p \) as explained variance during \textit{Compute}; monitor \( f(p) \) as a feedback signal in \textit{Inference}; rebalance model structure in \textit{Model} to steer \( p \) away from fragile extremes; and apply exploratory or exploitative interventions in \textit{Action} accordingly.

Thus, the golden-ratio partition becomes an internal rule of governance for the CIMA process, ensuring that the organization of knowledge and receptivity to surprise maintain the self-similarity and adaptive capacity characteristic of systems operating near criticality, while the exact maximum of \( f \) warns about the zone of maximal informational exposure.

\subsection{A minimal predictive--coding instantiation of the CIMA loop}
\label{subsec:minimal_instantiation}

To make the CIMA framework biologically explicit while preserving minimal assumptions, we now instantiate its components in a simple predictive--coding setting, representative of cortical microcircuits or small neuronal populations.

\paragraph{Generative environment.}
We consider a scalar sensory signal
\[
y_t = \theta_t + \eta_t,
\]
where $\theta_t$ denotes a latent environmental state and $\eta_t \sim \mathcal N(0,\sigma_t^2)$ represents sensory noise or environmental volatility. Variations in $\sigma_t^2$ define a natural and biologically interpretable class of perturbations, corresponding to changes in stimulus reliability, arousal, or contextual uncertainty.

\paragraph{Internal model and prediction error.}
The system maintains an internal estimate $\mu_t$ of the latent state and generates predictions $\hat y_t = \mu_t$. Prediction errors are given by
\[
\varepsilon_t = y_t - \hat y_t.
\]
Model updating follows a minimal predictive--coding rule,
\[
\mu_{t+1} = \mu_t + \alpha_t \, \omega_t \, \varepsilon_t,
\]
where $\alpha_t$ is a learning-rate-like parameter and $\omega_t$ denotes the precision weighting of prediction errors, commonly associated with synaptic gain modulation or neuromodulatory control in cortical circuits.

\paragraph{Operational definition of the informational balance.}
Within a sliding temporal window, we define the proportion of explained variance as
\[
p_t \;=\; 1 - \frac{\mathrm{Var}(\varepsilon_t)}{\mathrm{Var}(y_t)} \in (0,1),
\]
so that $p_t$ quantifies the fraction of sensory variability captured by the internal model, while $1-p_t$ represents residual novelty or unexplained information. This provides a concrete realization of the \emph{Compute} stage of the CIMA loop.

The informational balance function
\[
f(p_t) = -(1-p_t)\ln(1-p_t) + \ln p_t
\]
can then be evaluated continuously as a diagnostic signal in the \emph{Inference} stage, indicating whether residual uncertainty or overconfidence dominates the system’s informational exposure.

\paragraph{Perturbations and payoff.}
We define the perturbation class as changes in sensory volatility,
\[
\Pi := \sigma_t^2,
\]
including both gradual modulations and intermittent increases corresponding to unexpected environmental fluctuations. The system’s payoff is defined as learning progress,
\[
\Phi(\Pi) := \Delta E(\Pi) = E_{\mathrm{before}} - E_{\mathrm{after}},
\]
where $E$ denotes mean squared prediction error computed over fixed time windows before and after a perturbation. A system is antifragile with respect to volatility if $\Phi(\Pi)$ is convex, i.e.,
\[
\frac{d^2 \Phi}{d \Pi^2} > 0,
\]
meaning that larger perturbations induce disproportionately larger improvements in predictive performance.

\paragraph{Model regulation and action.}
In the \emph{Model} stage, parameters such as $\omega_t$ or $\alpha_t$ are adjusted to regulate $p_t$ away from fragile extremes. Targeting a corridor centered on the self-similar partition $p_\varphi = 1/\varphi \approx 0.618$ maintains a balance between exploitation (model consolidation) and exploration (openness to surprise). Finally, in the \emph{Action} stage, the system may modulate sampling, gain, or exposure to uncertainty—either amplifying perturbations when $p_t \ll p_\varphi$ or injecting controlled noise when $p_t \gg p_\varphi$—to remain near a dynamically critical regime.

This minimal instantiation demonstrates how the informational balance function, the golden-ratio partition, and the CIMA loop can be mapped onto measurable quantities in a biologically plausible predictive--coding system, thereby grounding the proposed route to antifragility in explicit neural and statistical terms.

\section{From balance to Antifragility}\label{sec:anti}

Recently it has been proposed that all complex adaptive systems (CAS) under their best configuration (health) tend to respond to their informational environment in the most antifragile way achievable. 

We define \textbf{antifragility} as the property by which a system not only withstands perturbations or stressors, but actually improves through them. Unlike robustness, which denotes resistance to change, or resilience, which refers to the ability to recover, antifragility implies a \emph{net gain} in function, adaptability, or structure as a result of exposure to uncertainty or volatility. It characterizes systems that benefit from disorder.

Formally, antifragility can be described by a triplet \( \langle S, \Pi, \Phi \rangle \), where \( S \) is the system, \( \Pi \) is a class of perturbations or stressors (such as shocks, noise, randomness, or fluctuations), and \( \Phi(\Pi) \) is a response or payoff function measuring the system’s outcome under those perturbations. A system is considered antifragile with respect to \( \Pi \) if the second derivative of its payoff function is positive: \( \frac{d^2 \Phi}{d \Pi^2} > 0 \). This convexity condition means that the system derives disproportionately greater benefit from larger perturbations than from smaller ones; in other words, the return curve bends upward, allowing the system to turn uncertainty into advantage.

This property is not absolute but contextual: antifragility depends on the nature of the perturbations, the time scale of observation, and the specific function used to evaluate performance or benefit. For instance, a system may be antifragile to variability in one domain but fragile in another, or antifragile in the short term but vulnerable over longer periods. It is therefore a relational and scale-dependent attribute, rather than a fixed trait.

Conceptually, antifragility operates across different dimensions. Qualitatively, it occupies a fourth category alongside fragility, robustness, and resilience. While fragile systems are damaged by stress, robust systems endure it, and resilient systems return to equilibrium, antifragile systems gain from it. This property can also be expressed across scales. Intrinsic antifragility arises from the internal structure of the system—such as redundancy, degeneracy, or modularity. Inherited antifragility is conferred by evolutionary or ecological context, where systems have been historically shaped by exposure to volatility. Induced antifragility results from intentional exposure to controlled stress, such as through training, hormesis, or exploratory experimentation.

Functionally, antifragility encompasses several behavioral classes. These include upregulation mechanisms where stress induces adaptive overcompensation; stochastic benefits where the system draws value from randomness or variance; scale-based advantages where larger disturbances produce disproportionate improvement; and temporal effects where the system may become stronger over time rather than degrade. Conversely, it excludes regimes of short-volatility fragility, where small, frequent perturbations accumulate harm, or temporal fragility, where irreversible losses accumulate due to aging or entropy.

Antifragile systems often exhibit the hallmarks of complex adaptive systems. These include nonlinear interactions, emergent behavior, self-organization, openness to information flow, and critical dynamics that balance order and disorder. Antifragility typically manifests near the edge of criticality, where systems remain sensitive to new signals without collapsing into noise.

Ultimately, antifragility is best understood as a general principle of adaptive organization in open systems. It formalizes how systems not only survive change, but \emph{use} it to become more responsive, more intelligent, and more capable of thriving in the face of unknowns. In this way, antifragility is the natural evolutionary response to complexity and uncertainty.

So, in simple words, antifragility refers to the property of a system that improves when exposed to perturbations, fluctuations, or uncertainty. This improvement is not linear, compensatory, or merely reactive. It reflects a structural capacity to transform disorder into advantage. 

This definition connects naturally with a broader narrative about the balance between what is known and what remains unknown. Let \( p \in (0,1) \) denote the fraction of reality that the system currently explains through its internal model. The complementary fraction, \( 1 - p \), represents the residual domain that is still unaccounted for, which is often the source of informational novelty and surprise. This contrast between known and unknown can be formalized through the function
\[
f(p) = -(1 - p)\ln(1 - p) + \ln p,
\]
which captures the net informational gain: the average surprise contributed by the unexplained component minus the verification cost of what is already known. As we have seen, this function has a unique maximum at \( p^{*} \approx 0.882 \), where the residual uncertainty contributes maximally to potential surprise. Independently, the self-similarity condition \( p:(1-p)=1:p \) singles out the golden-ratio partition \( p_\varphi = 1/\varphi \approx 0.618 \), where the ratios between known, unknown, and whole are recursively nested. At this point, the ratio between known and unknown equals the ratio of the known to the whole, forming a self-similar structure. This proportionality is not arbitrary. It reflects a recursive principle that aligns with scale-invariant dynamics commonly found in critical systems.

The curvature of \( f(p) \), which is concave on $(0,1)$, indicates that around its maximum small deviations in \( p \) reduce the net informational asymmetry: moving away from \( p^{*} \) decreases the excess surprise contributed by the unknown. Antifragility, however, does not rely on the convexity of \( f \) in \( p \), but on the convexity of the payoff function \( \Phi(\Pi) \) with respect to perturbations \(\Pi\). The role of \( f(p) \) is to quantify how the epistemic split between known and unknown modulates the potential for learning; the role of \( \Phi \) is to quantify how this potential is leveraged when the system is actually perturbed.

Within a CIMA loop, these two aspects are coupled. In the Compute phase, the system estimates the proportion \( p_t \) as the ratio of explained to total variance. This quantification reveals whether the model currently overfits or underfits reality. In the Inference phase, the system evaluates the value of \( f(p_t) \) and monitors dynamical signatures such as spectral slopes or avalanche distributions. These diagnostics indicate whether the system is maintaining the delicate balance required to stay near a critical state. The Model phase then adjusts internal structures to preserve a desirable proportion between known and unknown—possibly guided by the golden partition \( p_\varphi \) as a self-similar target. This may involve fine-tuning the complexity of representations, modifying the strength of priors, or adjusting the allocation of resources between exploration and exploitation. In the Action phase, the system interacts with its environment in order to realign itself with this informational balance. When \( p_t \) drifts too high, indicating overconfidence and the risk of fragility, the system may introduce controlled uncertainty or stressors. When \( p_t \) becomes too low, suggesting disintegration or loss of coherence, the system may reinforce learned patterns and stabilize its internal models.

This entire process is guided by the recursive proportion expressed by the identity \( 1/\varphi = \varphi - 1 \). The golden ratio thus becomes a principle of self-similarity across scales. It ensures that at every level of organization there remains enough informational openness to allow for adaptation, but also enough structure to preserve coherence. The known corresponds to the modular, redundant, and internally stable components of the system. The unknown arises from environmental, social, or temporal contexts that introduce variability. Together, they interact through a calibration process that continuously repositions the system within an antifragile corridor.

The CIMA loop, when properly calibrated, does not simply seek stability or resilience. It dynamically pursues an equilibrium between explanation and surprise. The golden-ratio partition \( p_\varphi \) is not just a mathematical artifact, but a generative principle that encodes a recursive relation between internal structure and external novelty. It provides a rule of internal governance that allows the system to remain adaptive, sensitive, and open to improvement. In this way, the balance function, the CIMA loop, and the principle of antifragility converge on a single operational insight: the systems that thrive are those that sustain a dynamic proportion between what they already know and what they are still capable of discovering.

\section{A Statistical Physics Interpretation: Balancing Mean Field and Fluctuations}

From the perspective of statistical physics, the balance function
\[
f(p) = -(1 - p)\ln(1 - p) + \ln p
\]
admits a natural reinterpretation as the point at which the contributions of the mean field and the fluctuations to the system’s information—or, analogously, to its statistical free energy—are in equilibrium. This idea echoes foundational intuitions in field theory and phase transitions, where a system's macroscopic behavior is shaped by the interplay between an average order parameter and the spectrum of its local deviations \cite{goldenfeld1992}.

In this analogy, the probability \( p \) represents the fraction of reality that the system currently explains, corresponding to the macroscopic order or mean field \( \langle \phi \rangle \). Conversely, the unexplained fraction \( 1 - p \) maps onto the fluctuations \( \delta \phi = \phi - \langle \phi \rangle \), which measure how the system deviates from its mean state. The term \( -p\ln p \) captures the cost of confirming the known, reflecting the deterministic part of the system's energy landscape, whereas \( -(1 - p)\ln(1 - p) \) expresses the average surprise contributed by the residual or unmodeled component, which parallels the entropy associated with fluctuations.

Rather than relying on an uncontrolled linearization of the logarithm, we can identify the golden-ratio partition directly from a self-similarity condition in the balance between mean field and fluctuations. Imposing that the relative weight of the mean field with respect to fluctuations equals the relative weight of the mean field with respect to the total,
\[
\frac{p}{1-p} = \frac{1}{p},
\]
leads to the quadratic equation
\[
p^2 + p - 1 = 0,
\]
with solution \( p = 1/\varphi \approx 0.618 \). At this point the informational contribution of the explained component and that of the unexplained component are related in a scale-free way: zooming from fluctuations to the whole preserves the same ratio.

This condition echoes the Ginzburg criterion, which delineates the boundary where the contributions from fluctuations become comparable to those from the mean field \cite{goldenfeld1992}. Beyond this point, simple mean-field theories break down and the system enters a critical regime where new, scale-invariant structures emerge.

Concrete physical models reflect this interpretation. In a $d$-dimensional Ising system, the total free energy decomposes as \( F = F_{\text{MF}} + F_{\text{fluct}} \). Above the critical dimension, the mean-field term dominates and the system behaves in an overly ordered fashion, which in informational terms corresponds to \( p \to 1 \). Far below the critical dimension, fluctuations overwhelm the mean behavior, and the system loses coherence, resembling the case where \( p \to 0.5 \). At the boundary where both terms balance, power laws and self-similarity emerge. This crossover can be associated with informational partitions such as \( p_\varphi \), where sensitivity and coherence coexist.

Similarly, in the Ginzburg–Landau formalism, the free energy includes a deterministic component \( F = a\phi^2 + b\phi^4 + \ldots \) and fluctuation corrections such as \( k_B T \ln \det[-\nabla^2 + m^2] \). The optimal regime occurs near \( a(T) \approx 0 \), where the effective mass \( m \to 0 \), allowing all modes to couple across scales. This gives rise to temporal and spatial self-similarity, exactly as seen in critical phenomena and in the $1/f$ spectra of neural dynamics.

In neurobiological systems, this balance manifests as a trade-off between synaptic gain, interpreted as a macroscopic control parameter, and synaptic noise, which captures microscopic variability. Properly tuning the gain ensures that neural oscillations do not freeze and that noise does not dissolve coherence. The resulting regime, characterized by scale-free avalanche activity and flexible responsiveness, is the biological counterpart of our informational balance.

This physical picture integrates seamlessly with the CIMA framework. In the Compute phase, the system measures the explained variance, which serves as a proxy for the mean field, and the residual variance, which maps onto fluctuations. In the Inference phase, it evaluates whether these contributions balance in the informational sense defined by \( f(p) \). When they do, the system operates near criticality. The Model phase adjusts the structural couplings—such as redundancy, buffering, or exposure mechanisms—so that the convexity of the payoff with respect to perturbations is preserved across time. In the Action phase, the system modulates its exposure to stress by introducing or removing perturbations as needed to maintain itself near the critical edge, where antifragility is maximized. At this point, noise is not simply tolerated but transformed into a functional resource due to the underlying convexity of the system’s response.

The exact optimum of \( f(p) \) at \( p^{*} \approx 0.882 \) represents the point where the unexplained component contributes maximally to the asymmetry between unknown and known surprise. Informational partitions like \( p_\varphi \) identify self-similar balances where mean-field and fluctuation contributions are proportionally nested. Together, these structures mark the onset of criticality in physical systems and, within the CIMA narrative, define the state in which the system becomes most capable of converting perturbations into functional gains. It is here that antifragility finds its physical signature.

\section{Discussion}

We have proposed a unifying informational framework in which the balance between prediction and surprise—formally captured by the function \( f(p) = -(1 - p)\ln(1 - p) + \ln p \)—acts as a central organizing principle for adaptive systems. At the heart of this model lie two distinct but complementary constructs: the exact maximum \( p^{*} \approx 0.882 \), which defines a point of maximal net excess surprise, and the golden-ratio partition \( p_\varphi = 1/\varphi \approx 0.618 \), which defines a self-similar partition between what the system explains and what remains surprising. The former marks where residual uncertainty has its strongest asymmetric impact; the latter encodes a recursive proportionality between knowledge and novelty.

\begin{figure}
    \centering
    \includegraphics[width=0.90\linewidth]{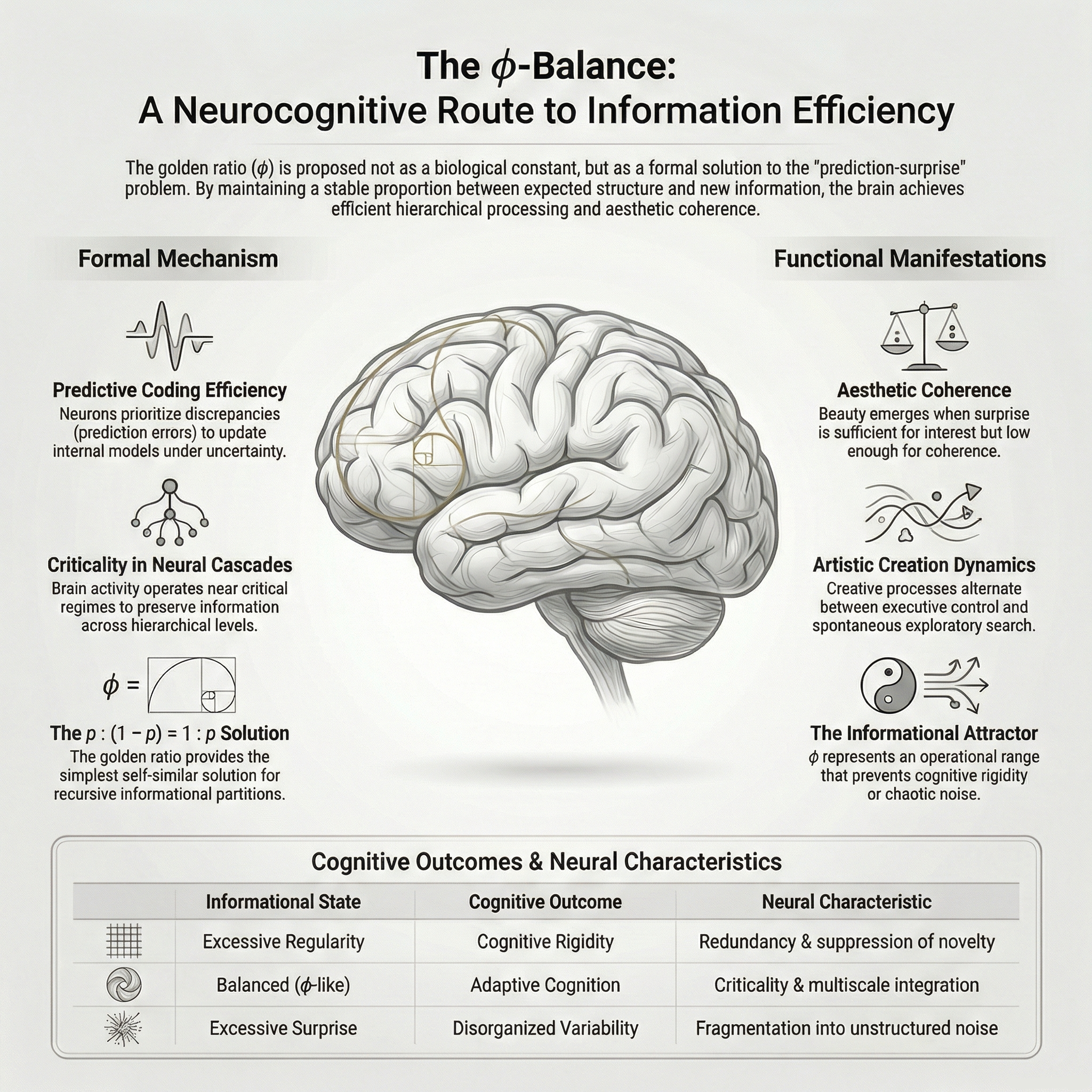}
    \caption{textbf{Self-similar informational partition and the $\varphi$-balance.} This diagram illustrates the self-similar partition defined by the relation $p:(1-p)=1:p,$ whose solution $p_\phi = \frac{1}{\varphi} \approx 0.618$ identifies a structurally distinguished informational split between explained and unexplained variance. Unlike the maximum of the balance function $$f(p)$$, this partition does not correspond to an extremum of informational imbalance, but to a recursive proportionality: the ratio of known to unknown equals the ratio of known to the whole. This self-similar structure preserves coherence under hierarchical recursion and provides a formal model of balanced predictive coding, in which internal models retain sufficient structure to guide inference while maintaining openness to residual novelty. The $\varphi$-partition therefore represents a candidate organizational principle for multiscale systems operating near criticality, where stability and variability coexist in scale-invariant proportions. For a deeper discussion see supplementary materials.
}
    \label{fig:placeholder}
\end{figure}

From a cognitive–neurological standpoint, this balance aligns with well-established mechanisms in brain function. The ratio \( p/(1 - p) \) corresponds to the relative precision assigned to internal predictions versus sensory prediction errors, a quantity dynamically modulated by neuromodulatory systems such as dopamine and noradrenaline. When the system drifts too far into certainty (\( p \to 1 \)), the capacity to learn from the environment diminishes; when it becomes overly sensitive to error (\( p \to 0 \)), coherence is lost. A corridor spanning intermediate values of \( p \), possibly structured by self-similar partitions like \( p_\varphi \), represents a region of maximal cognitive adaptivity—where model confirmation and model updating are optimally balanced. Empirical studies of resting-state brain activity support this idea, showing that power-law dynamics, maximal dynamic range, and high Fisher information co-occur near putative critical zones.

This theoretical structure finds operational form in the CIMA loop: Compute–Inference–Model–Action. This cycle enables a system to continuously assess the current explanatory state, estimate its distance from balanced partitions, recalibrate its internal model, and intervene either on itself or its environment to remain poised at the edge of criticality. In this configuration, the system transforms perturbations into learning, variance into structure, and surprise into antifragile gain.

We suggest that the golden-ratio partition is more than a mathematical curiosity. It is a candidate for a universal cognitive invariant—a principle of adaptive organization that links criticality, self-similarity, and antifragility. Systems that adhere to such a rule do not merely survive uncertainty; they cultivate it as a substrate for functional improvement. In doing so, they realize a biologically grounded form of intelligence that operates, not by eliminating error, but by orchestrating a dynamic tension between the known and the unknown.




\section*{Supplementary material}

# Supplementary materials for A golden-ratio partition of information and the balance between prediction and surprise: a neuro-cognitive route to antifragility

## A neurocognitive interpretation of φ-like partitions in perception and artistic emergence

*This Supplementary Note provides a theoretical and conceptual extension of the main manuscript. It does not introduce new empirical claims, measurements, or neurophysiological constants. Instead, it situates the formal results of the main text within established frameworks in theoretical neuroscience, predictive coding, and biological criticality, offering an interpretive scaffold rather than testable biological predictions.*

## Introduction

The recurrence of φ-like partitions across domains as diverse as psychological processing and artistic creation has often been described as an intuitive or aesthetic phenomenon. However, such recurrence can be addressed more rigorously when framed as a problem of information organization under conditions of uncertainty. Within predictive coding, cognition is not understood as a passive representation of the external world, but as an active inferential process through which the brain generates anticipatory models and continuously contrasts them with incoming sensory information, adjusting these models to minimize long-term prediction error without fully suppressing surprise (Friston, 2010; Clark, 2016).

From a neurocognitive perspective, this balance is not implemented by a single locus, but by distributed hierarchical circuits that integrate information across multiple temporal and spatial scales. Neurophysiological evidence suggests that, to sustain efficient multiscale integration, cortical activity tends to operate near critical regimes, where global stability and local variability coexist functionally. Such regimes allow information to propagate without collapsing into rigidity or degrading into noise, a condition that is central to hierarchical inference and adaptive cognition (Beggs & Plenz, 2003; Shew & Plenz, 2013).

Within this framework, the golden ratio is introduced here not as a geometric preference or biological constant, but as a candidate formalization of a functional compromise: a stable proportion between predictive structure and residual novelty that remains coherent under recursive processing. In this sense, φ can be interpreted as a candidate equilibrium point within a restricted class of self-similar informational partitions, where sufficient regularity sustains

coherent inference while sufficient uncertainty preserves sensitivity to emergent features. This interpretation aligns with approaches emphasizing informational efficiency and energetic constraints in neural processing, where optimal coding minimizes redundancy without sacrificing discriminative capacity (Laughlin, 2001; Friston, 2010).

Building on these premises, this Supplementary Note Ω develops an integrative trajectory:
 (1) it examines the prediction-surprise balance in predictive coding;
 (2) links this balance to criticality and scale invariance in hierarchical networks;
 (3) conceptualizes aesthetic perception as a case of predictive coherence under balanced complexity;
 (4) frames artistic creation as an extended inferential process alternating control and exploration; and
 (5) proposes φ as a **plausible formal attractor** of neural information flow under explicit constraints, offering a unified interpretive account of the convergence of perception, cognition, and artistic emergence (Clark, 2016; Parr et al., 2022).

## Predictive coding and the φ-balance

Within predictive coding frameworks, neuronal populations do not directly encode the external world, but instead prioritize the representation of discrepancies between expected and observed signals. These prediction errors constitute the central variable driving the updating of internal models, enabling perceptual and cognitive inference under uncertainty (Friston, 2010; Clark, 2016).

For this process to be efficient, it is not sufficient that prediction errors merely occur; the gain with which these errors influence model updating must be regulated. When error gain is excessive, neural activity becomes unstable and dominated by sensory fluctuations; when it is insufficient, the system becomes rigid, privileging prior expectations at the expense of sensitivity to incoming information. Inferential efficiency therefore depends on maintaining a stable distribution of informational variance between prediction and novelty (Hohwy, 2013).

From a formal standpoint, optimal coding can be characterized by three necessary conditions:
 (a) predictive signals explain a relatively stable proportion of total variance;
 (b) incoming information not captured by the model occupies the complementary proportion; and
 (c) the relation between these proportions is preserved under recursive hierarchical processing.
 Without such invariance, inferential coherence degrades as models scale, either by amplifying noise or by prematurely suppressing relevant information at higher hierarchical levels (Friston, 2010; Parr et al., 2022).

The self-similar relation

$$p:(1-p)=1:p$$

represents the **simplest non-trivial solution** satisfying these conditions under a minimal self-similarity constraint. Its solution corresponds to the golden ratio. Importantly, this relation should be understood as a **formal fixed point** of this constraint, not as a unique solution nor as an empirically imposed proportion. In this sense, φ functions as a minimal formal solution that preserves inferential coherence under hierarchical recursion, rather than as a biologically mandated constant (Friston, 2010; Laughlin, 2001).

## Criticality and hierarchical scaling

Efficient predictive coding cannot be sustained under arbitrary neural dynamics. For a stable distribution of variance between prediction and novelty to be preserved across hierarchical processing, neuronal networks must operate near critical regimes. Empirical studies of neural avalanches and scale-free connectivity show that cortical activity consistently occupies an intermediate regime between order and randomness, where stability and variability coexist functionally (Beggs & Plenz, 2003; Plenz & Thiagarajan, 2007; Shew & Plenz, 2013).

In critical regimes, neuronal correlations follow power-law distributions, allowing information to propagate across multiple scales without attenuation or explosive amplification. This property is essential for hierarchical inference: predictions generated at higher levels must remain relevant at lower levels, while local errors must be able to scale without losing informational value (Plenz & Thiagarajan, 2007; Deco & Jirsa, 2012).

From a formal perspective, criticality **can be interpreted** as a structural requirement for preserving scale invariance, which in turn constitutes a necessary condition for maintaining self-similar relations under recursive hierarchical processing. Without such invariance, any stable proportion between predictive structure and novelty would collapse as information traverses cortical levels. In this sense, criticality is not merely a byproduct of neural processing, but a plausible organizational regime enabling multiscale inferential coherence (Shew & Plenz, 2013).

## Aesthetic perception as predictive coherence

Aesthetic perception does not constitute a domain separate from ordinary cognition, but rather a particular case in which predictive coding mechanisms reach a regime of especially stable inferential coherence. Aesthetic experience emerges when the perceptual system integrates structural regularities and unexpected variations such that prediction error is reduced without being eliminated, preserving a controlled degree of uncertainty that sustains attentional engagement (Leder et al., 2004; Van de Cruys & Wagemans, 2011).

For this experience to arise, it is not sufficient for a stimulus to exhibit complexity or symmetry. The distribution of informational variance between what is predictable and what is novel must fall within a narrow functional range. Excessively regular stimuli concentrate variance in

prediction and generate perceptual redundancy, whereas highly unpredictable stimuli shift variance toward error, producing disorganization and loss of meaning (Van de Cruys, 2017).

Cognitive neuroscience studies indicate that, under these conditions, coordinated activation occurs across networks involved in meaning integration, salience attribution, and affective evaluation. Functional synchronization between the default-mode network and the salience network is associated with states in which surprise is sufficient to generate interest, but not so high as to prevent the formation of coherent expectations (Chatterjee & Vartanian, 2014; Vessel et al., 2019).

*Importantly, this account does not posit a specific numerical proportion governing aesthetic perception, but rather identifies a narrow functional range in which predictive coherence and novelty coexist.*

# Artistic creation as extended predictive coding

Artistic creation can be conceptualized as a limiting case of inferential exploration, in which general cognitive principles operate under conditions of heightened uncertainty and openness. Rather than constituting a rupture with ordinary inferential mechanisms, the creative process represents an intensification of these mechanisms, characterized by the deliberate introduction of variations, evaluation of deviations, and progressive reorganization of internal models (Friston, 2010).

Neuroimaging evidence supports this formulation by showing that, during artistic creation, there is a functional alternation between networks associated with executive control and networks implicated in spontaneous exploration, internal simulation, and free association. This dynamic reflects a flexible allocation of cognitive resources between control and openness, which is necessary to sustain the creative process in a stable and productive manner (Beaty et al., 2016).

From a formal perspective, this oscillation can be described as the preservation of a stable relationship between prediction and surprise throughout the creative process. Within this framework, a φ-like balance may serve as a **useful formal abstraction** of the optimal compromise between regularity and novelty: it allows each deviation to be integrated into an expanded model, favoring sequences of coherent transformations rather than abrupt or chaotic ruptures. Thus, φ is not proposed as a fixed empirical proportion, but as a formal attractor describing how creative systems maintain coherence while remaining open to emergence (Laughlin, 2001; Friston, 2010).

# Unified mechanism: φ as an attractor of neural information flow

If neurocognitive systems seek simultaneously to minimize irrelevant surprise, maximize meaningful novelty, and preserve hierarchical coherence, only certain distributions between prediction and error prove functional. Under these conditions, φ may be understood as a **formal attractor in the space of informational descriptions of neural inference**, rather than as a physiological or dynamical attractor in neural state space (Friston, 2010; Parr et al., 2022).

This convergence can be observed across perception, cognition, decision-making, and artistic creation at an interpretive level. In each domain, adaptive functioning deteriorates when informational flow becomes trapped in excessive redundancy or fragments into unstructured noise. A φ-like balance delineates an operational range in which information can propagate without collapse, allowing continuous adaptation without loss of coherence (Shew & Plenz, 2013; Clark, 2016).

From a neuropsychological perspective, persistent deviations from this regime are associated with cognitive rigidity, disorganized variability, or difficulties integrating information across multiple scales, reinforcing the view that such balances constitute functional conditions for adaptive cognition rather than fixed numerical laws (Stuss & Knight, 2013; Uddin, 2021).

# Conclusions

The trajectory developed in this Supplementary Note Ω articulates a coherent interpretive framework for the recurrent appearance of φ-like partitions in psychological processing and artistic creation. Rather than reflecting an arbitrary aesthetic preference, these partitions can be understood as expressions of neurocognitive principles governing efficient information management under uncertainty (Friston, 2010; Clark, 2016).

The available evidence suggests that such balances can be sustained only when neuronal networks operate near critical regimes in which stability and variability coexist functionally. In these contexts, φ emerges as an efficient **formal configuration** for preserving inferential coherence across hierarchical levels, not as an empirically fixed biological parameter (Beggs & Plenz, 2003; Shew & Plenz, 2013; Laughlin, 2001).

From this perspective, aesthetic perception and artistic creation can be conceived as particular expressions of a shared neurocognitive architecture. In both cases, meaning emerges when novelty is integrated without destroying coherence, and when structure flexes without collapsing into rigidity. Thus, φ appears here as a **formal signature of adaptive systems** capable of remaining stable while transforming.



\end{document}